\documentclass[12pt,leqno]{amsart}
\usepackage{amsmath, amssymb, amscd, amsfonts, xypic,   stmaryrd, turnstile, mathrsfs, eucal}

\usepackage
{hyperref}

\newtheorem{theorem}{Theorem}[section]
\newtheorem{lemma}[theorem]{Lemma}
\newtheorem{proposition}[theorem]{Proposition}
\newtheorem{corollary}[theorem]{Corollary}

\theoremstyle{definition}
\newtheorem{definition}[theorem]{Definition}
\newtheorem{example}[theorem]{Example}
\newtheorem{remark}[theorem]{Remark}

\def\car#1,#2,#3,#4{
$$
   \CD
   #1           @>{}>>          #2        \\
   @V{{}}VV                  @VV{{}}V  \\
   #3        @>{{}}>>   #4
   \endCD
   $$}

\def\Car#1,#2,#3,#4,#5,#6,#7,#8{
$$
   \CD
   #1           @>#2>>          #3        \\
   @V{#4}VV                  @VV{#5}V  \\
   #6        @>{#7}>>   #8
   \endCD
   $$}

\begin{document}

\title[ nagata extensions]{ Some more combinatorics results on Nagata Extensions}
\author[G. Picavet and M. Picavet]{Gabriel Picavet and Martine Picavet-L'Hermitte}
\address{Universit\'e Blaise Pascal \\
Laboratoire de Math\'ematiques\\ UMR6620 CNRS  \\ 24, avenue des Landais\\
BP 80026 \\ 63177 Aubi\`ere CEDEX \\ France}

\email{Gabriel.Picavet@math.univ-bpclermont.fr}
\email{picavet.gm(at)wanadoo.fr}

\begin{abstract}  We show that the length of a  ring extension $R\subseteq S$ is preserved under the formation of the Nagata extension $R(X)\subseteq S(X)$. A companion  result holds for the Dobbs-Mullins invariant. D. Dobbs and the authors proved elsewhere that the cardinal number of the set $[R,S]$ of subextensions of $R \subseteq S$ is preserved under the formation of Nagata extension when $|[R(X),S(X)]|$ is finite. We show that in the only pathological case, namely $R\subseteq S$ is subintegral, then $|[R,S]|$ is preserved if and only if it is either infinite or finite and $R\subseteq S$ is arithmetic; that is, $[R,S]$ is locally a chain.  The last section gives properties of arithmetic  extensions and their links with Pr\"ufer extensions.
\end{abstract}


\keywords  {arithmetic extension, FIP, FCP, FMC  extension, minimal extension, Pr\"ufer extension, integral extension, support of a module, Nagata ring, t-closure}

\maketitle

\section{Introduction and Notation}

We consider the category of commutative and unital rings and  first give some notation and definitions, needed for explaining the subject of the paper.
 Let $R\subseteq S$ be a (ring) extension. The set of all $R$-subalgebras of $S$ is denoted by $[R,S]$ and the integral closure of $R$ in $S$ by $\overline R$.  As usual, Spec$(R)$, Max$(R)$ and $\mathrm{Min}(R)$ are the sets of prime ideals, maximal ideals and minimal prime ideals of a ring $R$. Moreover, $\mathrm{Tot}(R)$ denotes the total quotient ring of a ring $R$. 

The support of an $R$-module $E$ is $\mathrm{Supp}_R(E):=\{P\in\mathrm{Spec}(R)\mid E_P\neq 0\}$, and $\mathrm{MSupp}_R(E):=\mathrm{Supp}_R(E)\cap\mathrm{Max}(R)$ is also the set of all maximal elements of $\mathrm{Supp}_R(E)$. If $E$ is an $R$-module, ${\mathrm L}_R(E)$ is its length. If $R\subseteq S$ is a ring extension and $P\in\mathrm{Spec}(R)$, then $S_P$ is both the localization $S_{R\setminus P}$ as a ring and the localization at $P$ of the $R$-module $S$. We denote by $(R:S)$ the conductor of $R\subseteq S$. Finally, $\subset$ denotes proper inclusion and $|X|$ the cardinality of a set $X$.

The extension $R\subseteq S$ is said to have FIP (for the ``finitely many intermediate algebras property") if $[R,S]$ is finite. A {\it chain} of $R$-subalgebras of $S$ is a set of elements of $[R,S]$ that are pairwise comparable with respect to inclusion. An extension $R\subseteq S$ is called a {\it chained} extension if $[R,S]$ is a chain. We say that the extension $R\subseteq S$ has FCP (for the ``finite chain property") if each chain in $[R,S]$ is finite. It is clear that each extension that satisfies FIP must also satisfy FCP. Dobbs and the authors characterized FCP and FIP extensions \cite{DPP2}. Minimal (ring) extensions,  introduced by Ferrand-Olivier \cite{FO}, are an important tool of the paper. Recall that an extension $R\subset S$ is called {\it minimal} if $[R,S]=\{R,S\}$. The key connection between the above ideas is that if $R\subseteq S$ has FCP, then any maximal (necessarily finite) chain of $R$-subalgebras of $S$, $R=R_0\subset R_1\subset\cdots\subset R_{n-1}\subset R_n=S$, with {\it length} $n<\infty$, results from juxtaposing $n$ minimal extensions $R_i\subset R_{i+1},\ 0\leq i\leq n-1$.  For any extension $R\subseteq S$, the {\it length} of $[R,S]$, denoted by $\ell[R,S]$, is the supremum of the lengths of chains of $R$-subalgebras of $S$. It should be noted that if $R\subseteq S$ has FCP, then there {\it does} exist some maximal chain of $R$-subalgebras of $S$ with length $\ell[R,S]$ \cite[Theorem 4.11]{DPP3}.

In passing we also consider a condition weaker than FCP on an extension $R \subset S$, recently explored by Ayache and Dobbs in \cite{AD}: there is a finite maximal chain in $[R,S]$ from $R$ to $S$ (condition FMC for some authors). Then \cite[Theorem 4.12]{AD} combined with \cite[Proposition 4.2]{FO} and \cite[Theorem 4.2]{DPP2} yields the following result, which may be useful to detect an FCP extension. 

\begin{proposition}\label{1.0} Let $R\subset S$ be an FMC extension of rings. Then $R\subset S$ satisfies FCP if and only the length of the $R$-module $\overline R/R$ is finite, or equivalently $R\subseteq \overline R $ has FCP.

\end{proposition}

We note here that $\overline R \subseteq S$ has FIP when $R\subseteq S$ has FCP \cite[Theorem 6.3]{DPP2}.

Let $R$ be a ring and $R[X]$ the polynomial ring in the indeterminate $X$ over $R$. (Throughout, we use $X$ to denote an element that is indeterminate over all relevant coefficient rings.) Also, let $C(p)$ denote the content of any polynomial $p(X)\in R[X]$. Then $\Sigma_R:=\{p(X)\in R[X]\mid C(p)=R\}$ is a
saturated multiplicatively closed subset of $R[X]$, each of whose elements is a non-zero-divisor of $R[X]$. The {\it Nagata ring of} $R$ is defined to be $R(X):= R[X]_{\Sigma_R}$. 

Let $R\subseteq S$ be an extension. It was shown in \cite[Theorem 3.9]{DPP3} that $R(X)\subseteq S(X)$ has FCP if and only if $R\subseteq S$ has FCP. One aim of this paper is to show that, when $R\subseteq S$ has FCP, then $\ell[R,S]=\ell[R(X),S(X)]$, a question addressed in \cite[Remark 4.18(b)]{DPP3}.

We begin to show that this property holds for FCP field extensions in Section 2. The main result is gotten in Section 3 where, after several steps involving the integral closure and the t-closure of an FCP extension, we prove in Theorem~\ref{3.3} that, when $R\subseteq S$ has FCP, then $\ell[R,S]=\ell[R(X),S (X)]$. We also introduce  the Dobbs-Mullins  invariant of an extension $R \subseteq S$ as  being the supremum $\Lambda(S/R)$ of the lengths  of residual extensions of $R\subseteq S$, considered as ring extensions  \cite{DM}. We show in Theorem~\ref{3.7} that $\Lambda(S/R) = \Lambda(S(X)/R(X))$.

We will have to consider the following material.

\begin{definition}\label{1.3} Let $R\subseteq S$ be an integral extension. Then $R\subseteq S$ is called {\bf infra-integral} \cite{Pic 2} (resp$.$ {\bf subintegral} \cite{S}) if all its residual extensions $R_P/PR_P\to S_Q/QS_Q$, (with $Q\in\mathrm {Spec}(S)$ and $P:=Q\cap R$) are isomorphisms (resp$.$ and the spectral map $\mathrm{Spec}(S)\to \mathrm{Spec}(R)$ is bijective). An extension $R\subseteq S$ is called {\bf t-closed} (cf. \cite{Pic 2}) if the relations $b\in S,\ r\in R,\ b^2-rb\in R,\ b^3-rb^2\in R$ imply $b\in R$. The $t$-{\bf closure} ${}_S^tR$ of $R$ in $S$ is the smallest $R$-subalgebra $B$ of $S$ such that $B\subseteq S$ is t-closed and the greatest $ B'\in[R,S]$ such that $R\subseteq B'$ is infra-integral.

The canonical decomposition of an arbitrary ring extension $R\subset S$ is $R \subseteq {}^+_S R \subseteq {}_S^tR \subseteq \overline R \subseteq S$, where $ {}^+_S R$ is the seminormalization of $R$ in $S$ (see \cite{S}).
\end{definition}

The other aim is achieved in Section 4. It consists to improve a characterization of the transfer of the FIP property for subintegral extensions of Nagata rings (see \cite[Theorem 3.30]{DPP3}). We consider only this (pathological) case because in the canonical decomposition of a ring extension, the subintegral part $R \subseteq {}^+_S R $ is the only obstruction for $R(X)\subseteq S(X)$ having FIP \cite[Theorem 3.21]{DPP3}.  This leads us to introduce extensions $R\subseteq  S$ such that $R_M \subseteq S_M$ is a chained extension for each $M\in\mathrm{Supp}_R(S/R)$. Such extensions are called {\it arithmetic}, the definition being reminiscent of arithmetic rings. Note that $\mathrm{Supp}(S/R)$ can be replaced with one of the following subsets $\mathrm{Spec}(R)$, $\mathrm{Max}(R)$, $\mathrm{MSupp}(S/R)$), since the natural map $[R,S] \to [R_P,S_P]$ is surjective for each $P\in \mathrm{Spec}(R)$.  We show in Theorem~\ref{4.3} that if $R\subset S$ is a subintegral extension, then  $R(X)\subset S(X)$ has FIP if and only if  $R\subset S$ has FIP and is arithmetic.

For an FCP extension $R\subseteq S$, it will be convenient to consider $\mathrm{MSupp}(S/R)$. Observe that an FCP extension $R\subseteq S $ is arithmetic if and only if $R_M\to S_M$ can be factored into a unique finite sequence of minimal morphisms, for each $M\in \mathrm{MSupp}(S/R)$. 

 Moreover, if $R\subseteq  T \subseteq S$ is an arithmetic extension, then so are $R\subseteq T$ and $T\subseteq S$.  Let $R\subseteq S$ be an extension with conductor $C:= (R:S)$. It is clear that $R\subseteq S$ is arithmetic if and only if $R/C\subseteq S/C$ is arithmetic.

The paper ends with Section 5, that contains  results on arithmetic extensions.

The following notions and results are also deeply involved in our study.

\begin{theorem}\label{1.2}\cite[Th\'eor\`eme 2.2 and Lemme 3.2]{FO} Let $A\subset B$ be a minimal extension. Then, there is some $M\in\mathrm{Max}(A)$, called the {\bf crucial (maximal) ideal} of $A\subset B$, such that $A_P=B_P$ for each $P\in\mathrm{Spec}(A)\setminus\{M\}$. We denote this ideal $M$ by $\mathcal{C}(A,B)$.

Moreover, $A\subset B$ is either an integral (finite) extension, or a flat epimorphism, these two conditions being mutually exclusive.
\end{theorem}

There are three types of minimal integral extensions, given by the following theorem.

\begin{theorem}\label{1.4} \cite [Theorem 2.2]{DPP2} Let $R\subset T$ be an extension and set $M:=(R: S)$. Then $R\subset T$ is minimal and finite if and only if $M\in\mathrm{Max}(R)$ and one of the following three conditions holds:

(a) {\bf inert case}: $M\in\mathrm{Max}(T)$ and $R/M\to T/M$ is a minimal field extension;

(b) {\bf decomposed case}: There exist $M_1,M_2\in\mathrm{Max}(T)$ such that $M= M _1\cap M_2$ and the natural maps $R/M\to T/M_1$ and $R/M\to T/M_2$ are both isomorphisms;

(c) {\bf ramified case}: There exists $M'\in\mathrm{Max}(T)$ such that ${M'}^ 2 \subseteq M\subset M',\  [T/M:R/M]=2$, and the natural map $R/M\to T/M'$ is an isomorphism.

Decomposed and ramified minimal extensions are infra-integral while inert minimal extensions are not.  Ramified minimal extensions are subintegral.
\end{theorem}

The next lemma will be used later. Let $\mathcal{P}$ be a property holding for a class $\mathcal C$ of ring extensions, stable under subextensions (i.e. $R \subseteq S$ in $\mathcal C$ and $[U,V] \subseteq [R,S]$ imply $U \subseteq V$ in $\mathcal C)$. We say that $\mathcal P$ admits a closure in $\mathcal C$ if the following conditions (i), (ii), (iii) and (iv) hold for any extension $R\subset S$ in $\mathcal C$:

(i) For any tower of extensions $R\subseteq U\subseteq S$, then $R\subseteq S$ has $\mathcal{P}$ if and only if $R\subseteq U$ and $U\subseteq S$ have $\mathcal{P}$.

(ii) There exists a largest subextension $T\in[R,S]$ such that $R\subseteq T$ has $\mathcal{P}$. 

(iii) No subextension $U\subseteq V$ of $T\subseteq S$ has   $\mathcal{P}$. 

(iv) $T=R$ when $R\subset S$ is a composite of finitely many minimal extensions which do not satisfy $\mathcal{P}$ .

Such a $T$ is unique, is  called the $\mathcal{P}$-closure of $R$ in $S$ and is  denoted  by $R^{\mathcal{P}}$. Some instances are  the separable closure in the class of  algebraic field extensions and  the t-closure in the class of integral ring extensions.

\begin{lemma} \label{1.5} Let $\mathcal{P}$ be a property of ring extensions admitting a $\mathcal P$-closure  in a class $\mathcal C$ of ring extensions. If an FCP extension $R\subseteq S$ belongs to  $\mathcal C $ and $R^{\mathcal{P}}$ is its $\mathcal{P}$-closure,  then, $\ell[R,S]=\ell[R,R^{\mathcal{P}}]+\ell[R^{\mathcal{P}},S]$. 
\end{lemma}

\begin{proof}  Obviously, $\ell[R,S]\geq\ell[R,R^{\mathcal{P}}]+\ell[R^{\mathcal{P}},S]$. We prove by induction on $n:=\ell[R,S]\geq 1$ that there exists a maximal chain from $R$ to $S$ with length $n$ containing $R^{\mathcal{P}}$. If $n=1$, then $R\subset S$ is a minimal extension, so that either $R^ {\mathcal{P}}=R$, or $R^{\mathcal{P}}=S$. Assume now that $n>1$ and that the induction hypothesis holds for any $n'<n$. We may assume that $R\neq R^{\mathcal{P}}$. Let $R=R_0\subset R_1\subset\cdots\subset R_{n-1}\subset R_n=S$ be a maximal chain of subextensions with length $n$.  The induction hypothesis applied to the extension $R_1\subseteq S$ (with length $n-1$) gives that there exists a maximal chain $R_1=R'_1\subset\cdots\subset R'_{n-1}\subset R'_n=S$ with length $n-1$ containing $R_1^{\mathcal{P}}$, so that $R_1\subseteq R_1^{\mathcal{P}}$ satisfies $\mathcal{P}$.

If $R\subset R_1$ satisfies $\mathcal{P}$, then, $R\subset R_1\subseteq R_1^{\mathcal{P}}$ satisfies $\mathcal{P}$ by (i), so that $R_1^{\mathcal{P}}=R^{\mathcal{P}}$. It follows that we get a maximal chain from $R$ to $S$ with length $n$ containing $R^{\mathcal{P}}$. 

Assume that $R\subset R_1$ does not satisfy $\mathcal{P}$. If $R_1\neq R_1^{\mathcal{P}}$, then $R_1 \subset R'_2$ satisfies $\mathcal{P}$ because $R'_2\subseteq R_1^{\mathcal{P}}$. Let $R'$ be the $\mathcal{P}$-closure of the extension $R\subset R'_2$. We have $R'\neq R'_2$ because $R\subset R'_2$ does not satisfy $\mathcal{P}$ by (i). Assume that $R\neq R'$. Because of the length of $R'_2 \subset\cdots\subset R'_{n-1}\subset R'_n=S$, we get that $R\subset R' $ is minimal and satisfies $\mathcal{P}$. For the same reason, $R'\subset R_2' $ is minimal. Let $R''$ be the $\mathcal{P}$-closure of the extension $R'\subset S$ (with length $n-1$). We have $R''=R^{\mathcal{P}}$. The induction hypothesis gives that there exists a maximal chain from $R'$ to $S$ with length $n-1$ containing $R''$, so that there exists a maximal chain from $R$ to $S$ with length $n$ containing $R^{\mathcal{P}}$. Now, assume that $R=R'$. By (iii), no subextension of $R\subset R'_2$ satisfies $\mathcal{P}$, a contradiction, since $R_1\subset R'_2$ satisfies $\mathcal{P}$. 

At last, assume that $R_1=R_1^{\mathcal{P}}$, then, $R=R^{\mathcal{P}}$ by (iv). Indeed, $R\subset S$ is composed of minimal subextensions, each of them not satisfying $\mathcal{P}$. 

To end, $\ell[R,S]=\ell[R,R^{\mathcal{P}}]+\ell[R^{\mathcal{P}},S]$. 
\end{proof}

We recover in particular that $\ell [R,S] = \ell[R,\overline R] + \ell[\overline R,S]$ \cite[Theorem 4.11]{DPP3}.

\begin{remark}\label{1.6} For the reverse order, there is some  companion result that can be  written if  after all  it  reveals useful. 
\end{remark}

We end by recalling some useful characterizations of the support of an   FCP extension.

\begin{lemma} \label{1.8} \cite[Remark 6.14 (b), Theorem 6.3]{DPP2} Let $R\subseteq S$ be an integrally closed  FMC extension. Then,  $\mathrm{Supp}(S/R)=\{P\in\mathrm{Spec}(R)\mid PS=S\}$.
\end{lemma}

\begin{lemma}\label{1.9}Ê\cite[Corollary 3.2]{DPP2} Suppose that there is a maximal chain $R=R_0 \subset\cdots\subset R_i \subset\cdots \subset R_n=S$ of extensions, where $R_i\subset R_{i+1}$ is minimal with crucial ideal $M_i$ for each $i = 0,\ldots , n-1$ ({\it i.e.} $R \subseteq S$ has FMC). Then $\mathrm {Supp}(S/R)$ is a finite set; in fact, $\mathrm {Supp}(S/R)=\{M_i\cap R\mid i=0,\ldots,n-1\}$.
\end{lemma}
\section {Preliminary results about FCP field extensions}

We  first observe that an FCP field extension $K \subset L$ is finite, whence algebraic. It follows that $L(X) \simeq L\otimes_K K(X)$ by \cite[Lemma 3.1]{DPP3}, so that $[L:K] = [L(X):K(X)]$. Moreover, a minimal field extension is clearly either separable, or purely inseparable (see for instance \cite{Ph}) and the degree of a minimal purely inseparable extension of a field $K$ is equal to the characteristic of $K$.

\begin{proposition} \label{2.1} Let $K\subset L$ be an FCP field extension and let $K_s$ be the separable closure of $K$ in $L$. Then,  $\ell[K,L]=\ell[K,K_s]+\ell[K_s,L]$. 
\end{proposition}
\begin{proof} We use Lemma~\ref{1.5}, where $\mathcal{P}$ is the property to be a separable extension and $K^{\mathcal{P}}=K_s$ is the separable closure (\cite[Ch. V, Proposition 13, p. 42]{Bki A}). 
\end{proof}

So, it is enough to consider the situation for FCP separable extensions and FCP purely inseparable extensions.

\begin{proposition} \label{2.2} Let $K\subset L$ be an FCP separable field extension. Then, $\ell[K,L]=\ell[K(X),L(X)]$. 
\end{proposition}
\begin{proof} Since $K\subset L$ has FCP, its degree  is finite. As a finite separable extension has a primitive element, it has FIP. We infer from \cite[Propositions 9 and 11]{DPP4}  that there is an order-isomorphism $[K,L]\to[K(X),L(X)]$, given by $T\mapsto T(X)$. It follows that any maximal chain of $K\subset L$ leads to a maximal chain of $K(X)\subset L(X)$. Conversely, any maximal chain of $K(X)\subset L(X)$ comes from a maximal chain of $K\subset L$, giving $\ell[K,L]=\ell[K(X),L(X)]$. 
\end{proof}

\begin{proposition} \label{2.3} Let $K\subset L$ be an FCP purely inseparable field extension. Then, $\ell[K,L]=\ell[K(X),L(X)]$. 
\end{proposition}
\begin{proof} Since $K\subset L$ is an FCP purely inseparable field extension, $K$ is a field of characteristic a prime number $p$ and $[L:K]$ is a power of $p$, say $p^n$. It follows that there is only one maximal chain composing $K\subset L$, and it has length $n$, and leads to a maximal chain composing $K(X)\subset L(X)$ with length $n$, which is also purely inseparable, with $K(X)$ of characteristic $p$ and $[L(X):K(X)]= p^n$. Then, $n=\ell[K,L]=\ell[K(X),L(X)]$. 
\end{proof}

\begin{proposition} \label{2.4} Let $K\subset L$ be an FCP field extension and let $K_s$ be the separable closure of $K$ in $L$. Then,  $K_s(X)$ is the separable closure of $K(X)$ in $L(X)$. 
\end{proposition}
\begin{proof} We got in the proof of Proposition~\ref{2.2} that $K\subseteq K_s$ has FIP, and so has a primitive element $\alpha$, which is separable over $K$. Then, $\alpha$ is also a primitive element of the extension $K(X)\subset K_s(X)$, and is separable over $K(X)$. It follows that $K(X)\subset K_s(X)$ is a separable extension. 

Moreover, $K_s\subseteq L$ is purely inseparable. Then, any element of $L$ is purely inseparable over $ K_s$, so that any element of $L(X)$ is purely inseparable over $K_s(X)$. Hence, $K_s(X)\subseteq L(X)$ is purely inseparable, giving that $K_s(X)$  is the separable closure of $K(X)$ in $L(X)$. 
\end{proof}

We can now state the result for FCP fields extensions.

\begin{theorem} \label{2.5} Let $K\subset L$ be an FCP field extension. Then, $\ell[K,L]=\ell[K(X),L(X)]$. 
\end{theorem}
\begin{proof} Let $K_s$ be the separable closure of $K$ in $L$. Then, $K_s(X)$ is the separable closure of $K(X)$ in $L(X)$ by Proposition~\ref{2.4}. Applying Proposition~\ref{2.1} twice, Proposition~\ref{2.2} and Proposition~\ref{2.3}, we get that $\ell[K(X),L(X)]=\ell[K(X),K_s(X)]+\ell[K_s(X),L(X)]=\ell[K,K_s]+\ell[K_s,L]=\ell[K,L]$.
\end{proof}

This gives the  result needed for the next section.

\begin{corollary}\label{2.6} Let $R\subset S$ be an FCP t-closed extension. Then, $\ell[R,S]=\ell[R(X),S (X)]$. 
\end{corollary}
\begin{proof} From \cite[Lemmata 3.3 and 3.15]{DPP3}, we get that $R(X)\subset S(X)$ is an FCP t-closed extension, with $\{MR(X)\mid M\in\mathrm{MSupp}(S/R)\}= \mathrm{MSupp}(S(X)/R(X))$. Then, \cite[Proposition 4.6 and Lemma 3.16 proof]{DPP3} give that $\ell[R(X),S(X)]=\sum [\ell [R_M(X),S_M(X)] | M\in \mathrm {MSupp}(S/R)] = \sum [\ell [R_M,S_M] | M\in \mathrm {MSupp}(S/R)]=\ell[R,S]$.
Hence we can reduce the proof to the case of a quasi-local ring $(R,M)$. Since  $M=(R:S)\in\mathrm{Max}(S)$ by \cite[Lemma 3.17]{DPP3}, we get $MR(X)=(R(X):S(X))= MS(X)$.  Now  $\ell[R(X),S(X)]=\ell[R(X)/MR(X),S(X)/MR(X)]=$  $\ell [(R/M)(X),(S/M)(X)]$ are  consequences of \cite[Proposition 3.7]{DPP2}. We then observe that  $\ell[R,S]= \ell[R/M,S/M]=\ell[(R/M)(X),(S/M)(X)]$ in view of Theorem~\ref{2.5} and \cite[Proposition 3.7]{DPP2}.
\end{proof}

\section{ On the lengths of FCP extensions of Nagata rings}

To introduce this section, we give the following lemma.

\begin{lemma}\label{3.1} Let $R \subset S$ be an integral extension and consider a maximal chain $\mathcal C$ of $ R$-subextensions of $S$ defined by $R=R_0\subset\cdots\subset R_i\subset\cdots\subset R _n= S$, where $R_i\subset R_{i+1}$ is  minimal for each $i\in\{0,\ldots,n-1\}$. Then, 

\begin{enumerate}
\item $R\subset S$ is infra-integral if and only if, for each $i\in\{0,\ldots,n-1\}$,  $R_i\subset R_{i+1}$ is either ramified or  decomposed. 

\item  $R \subset S$ is t-closed if and only if $R_i\subset R_{i+1}$ is inert for each $i\in\{0,\ldots,n-1\}$. 

\item If in addition $(R,M)$ is a quasi-local ring and the conditions of (2) hold, then $M= (R:S)$ and $(S,M)$ is a quasi-local ring.
\end{enumerate}
\end{lemma}

\begin{proof} (1) is obvious, because all the residual field extensions are isomorphisms.

(2) Assume that $R\subset S$ is t-closed. Then $R_i\subset R_{i+1}$ is inert for each $i\in\{0,\ldots,n-1\}$ in view of \cite[Lemma 5.6]{DPP2}. Conversely, if $R_i\subset R_{i+1}$ is inert for each $i\in\{0,\ldots,n-1\}$, and so t-closed, then $R \subset S$ is obviously t-closed. 

(3) Moreover, if $(R,M)$ is quasi-local, \cite[Lemma 3.17]{DPP3} shows that $M$ is the only maximal ideal of $S$. 
 \end{proof}
 
 We can now see  how the t-closure is involved in the length of an integral FCP extension.

\begin{proposition}\label{3.2} Let $R\subset S$ be an integral FCP extension,  then $\ell[R,S]=\ell[R,{}_S^tR]+\ell[{}_S^tR,S]$.
\end{proposition} 

\begin{proof} Use Lemma~\ref{1.5} and Lemma~\ref{3.1}, where $\mathcal{P}$ is the property to be an infra-integral extension, and $R^{\mathcal{P}}={}_S^tR$ is the t-closure of $R$ in $S$.
\end{proof}

We are now in position to give a positive answer to \cite[Remark~4.18(b)]{DPP3}.

\begin{theorem}\label{3.3} Let $R\subset S$ be an FCP extension. Then, $\ell[R,S]=\ell[R(X),S(X)]$.
\end{theorem} 

\begin{proof} Let $R\subset S$ be an FCP extension. We begin to notice that the t-closure of $R(X)$ in $S (X)$ is ${}_{S(X)}^tR(X)=({}_S^tR)(X)$ by \cite[Lemma 3.15]{DPP3}. Moreover, in \cite[Remark 4.18 (b)]{DPP3}, we proved that $\ell[R,S]=\ell[R(X),S(X)]$ if and only if $\ell[R,\overline R]=\ell[R(X),\overline R (X)]$. It follows that we can assume that $R\subset S$ is an integral FCP extension. But, Proposition~\ref{3.2} gives that $\ell[R,S]=\ell[R,{}_S^tR]+\ell[{}_S^tR,S]$, and, in the same way, $\ell[R(X),S(X)]=\ell[R (X),({}_S^tR)(X)]+\ell[({}_S^tR)(X),S(X)]$. Now, $\ell[{}_S^tR,S]=\ell[({}_S^tR)(X),S(X)]$ by Corollary~\ref{2.6}. To end, $\ell[R,{}_S^tR]=\ell[R(X),({}_S^tR)(X)]$  \cite[Proposition 4.7]{DPP3}.
\end{proof}

\begin{corollary}\label{3.4} Let $R\subset S$ be an FCP extension and $n$ a positive integer. Then, $\ell[R,S]=\ell[R(X_1,\ldots,X_n),S(X_1,\ldots,X_n)]$.
\end{corollary} 

We end this section by some considerations about the length of FCP extensions $R\subseteq S$ with respect to their residual extensions. Following Dobbs and Mullins \cite{DM}, we define $\Lambda(S/R)$ to be the supremum of the lengths of residual extensions of $R\subseteq S$, considered as ring extensions.

\begin{proposition}\label{3.5} Let $R\subset S$ be an FCP extension. Then $\Lambda(S/R)=\Lambda (\overline R/{}_S^tR)$.
\end{proposition}

\begin{proof} We first observe that an FCP extension $R \subseteq S$ is strongly affine, that is each of the $R$-algebras $T \in [R,S]$ is of finite type. Since $R \subseteq S$ is a composite of minimal morphisms that are either flat epimorphisms or integral morphisms, $R \subseteq T$ is an INC extension for $T \in [R, S]$ and hence a quasi-finite extension. Moreover, the residual extensions of each minimal morphism $T \subset U$, with $T, U \in [R,S]$ are either isomorphisms or minimal field extensions, induced by inert minimal morphisms. Then in the canonical decomposition $R \subseteq {}_S^tR \subseteq \overline R \subseteq S$, the extension $\overline R \subseteq S$ is a flat epimorphism by the Zariski Main Theorem. Therefore the residual extensions of $R \subseteq S$ identify with the residual extensions of ${}_S^tR \subseteq \overline R$ and  the components of maximal chains in $[{}_S^tR,\overline R]$ need to be minimal inert extensions by Lemma \ref{3.1}(2). The above discussion shows that for an FCP extension $R \subseteq S$, then $\Lambda (S/R) = \Lambda (\overline R/{}_S^tR)]$.
\end{proof}

So, it is enough to consider an FCP integral t-closed extension $R\subset S$. 

\begin{proposition}\label{3.6} Let $R\subset S$ be an FCP integral t-closed extension. Then $\Lambda (S/R)=\sup_{M\in\mathrm{MSupp}(S/R)}\ell[R_M,S_M]$ and $\ell[R,S]\leq n\Lambda (S/R)$, where $n:=|\mathrm{MSupp}(S/R)|$.
\end{proposition}

\begin{proof}
We get $\sum [\ell[R_M,S_M] | M\in\mathrm{MSupp}(S/R)] = \ell[R,S]\ (*)$ by  \cite[Proposition 4.6]{DPP3}. 
Assume first that $(R,M)$ is a quasi-local ring, and so $(R:S)=M$ by Lemma \ref{3.1}. Then, $\ell[R,S]=\ell[R/M,S/M]=\Lambda(S/R)$ by \cite[Proposition 3.7]{DPP2}. Now, in the general case, set $\mathrm{MSupp}(S/R):=\{M_1,\ldots,M_n\}$. Consider a maximal chain of $R$-subextensions of $S$ defined by $R=R_0\subset\cdots\subset R_i\subset\cdots\subset R_p=S$, where $R_i\subset R_{i+1}$ is minimal inert for each $i\in\{0,\ldots,p-1\}$. In view of Lemma \ref{1.9}, we have, $\{M_1,\ldots,M_n\}=\{\mathcal C(R_i,R_{i+1})\cap R\mid i\in\{0, \ldots,p-1\}\}=\{(R_i:R_{i+1})\cap R\mid i\in\{0,\ldots,p-1\}\}$. An easy induction using \cite[Lemma 3.3]{DPP2}, shows that we can exhibit $R$-subextensions of $S$ such that $R=R'_0\subset\cdots\subset R' _j\subset\cdots\subset R'_n=S$,  $R'_j\subset R'_{j+1}$ is t-closed for each $j\in\{0,\ldots,n-1\}$ and satisfies $(R'_j:R'_{j+1})\cap R=M_{j+1}$. This is obvious for $j=0$. But, since $R\subset R'_1$ is t-closed and integral, for each $j\in\{2,\ldots,n\}$, there is a unique $M'_j\in\mathrm{Max}(R'_1)$ lying above $M_j$, and we have $\mathrm{MSupp}(S/R'_1)=\{M'_2,\ldots,M'_n\}$. Then, for each $j\in\{1,\ldots,n\}$, we have $[R_ {M_j},S_{M_j}]=[(R'_{j-1})_{M_j},(R'_j)_{M_j}]$ since $R_{M_j}=(R'_{j-1})_{M_j}$ and $(R'_j)_{M_j}=S_{M _j}$. It follows that $\ell[R_{M_j},S_{M_j}]=\Lambda(S_{M_j}/R_{M_j})$, so that $\ell[R,S]=\sum_{j=1}^n\ell [R_{M_j},S_{M_j}]=\sum_{j=1}^n\Lambda(S_{M_j}/R_{M_j})=\sum_{j=1}^n\Lambda((R'_j)_{M_j}/(R'_{j-1}) _{M_j})=\sum_{j=1}^n\Lambda(R'_j/R'_{j-1})$ because $(R'_j)_{M}=(R'_{j-1})_{M}$ for any $M\neq M_j$. To end, let $Q\in\mathrm{Spec}(S)$ and set $P:=Q\cap R$. If $P\not\in\mathrm{MSupp}(S/R)$, we get that $R_P=S_P=S_Q$, so that $k(P)=k(Q)$. If $P\in\mathrm{MSupp}(S/R)$, then $Q$ is the only prime ideal of $S$ lying over $P$, so that $S_P=S_Q$ and $[k(P),k(Q)]=[R/P,S/Q]$. It follows that $\Lambda(S/R)=\sup_{j\in\{1,\ldots,n\}}\Lambda(S_{M_j}/R_{M_j})=\sup_{j\in\{1,\ldots,n\}}\ell[R_{M_j},S_{M_j}]$. 

For each $M\in\mathrm{MSupp}(S/R)$, we have $ \ell[R_M,S_M]=\Lambda(S_M/R_M)\leq \Lambda(S/R)$, so that $(*)$ gives $\ell[R,S]\leq n\Lambda (S/R)$.
\end{proof}

Coming back to the Nagata ring extension, we get the following theorem.

\begin{theorem}\label{3.7} Let $R\subset S$ be an FCP extension. Then $\Lambda (S/R)=\Lambda (S(X)/R(X))$.
\end{theorem}

\begin{proof}  We get
$\Lambda(S(X)/R(X))=\Lambda (\overline {R(X)}/{}_{S(X)}^tR(X))$ and $\Lambda(S/R)=\Lambda (\overline R/{}_S^tR)$  from Proposition \ref{3.5} and $\overline {R(X)}=\overline R(X)$ and ${}_{S(X)}^tR(X)=({}_S^tR)(X)$ from \cite[Proposition 3.8 and Lemma 3.15]{DPP3}. To make easier the reading,we  set $R':={}_S^tR$ and $S':=\overline R$. Proposition \ref{3.6} gives 
$$\Lambda (S'/R')=\sup_{M\in\mathrm{MSupp}(S'/R')}\ell[R'_M,S'_M]$$ 
and 
$$\Lambda (S'(X)/R'(X))=\sup_{M'\in\mathrm{MSupp}(S'(X)/R'(X))}\ell[R'(X)_{M'},S'(X)_{M'}]$$
 Now, we have the following results: $\ell[R'_M,S'_M]=\ell[R'_M(X),S'_M(X)],$
 
\noindent $\mathrm{MSupp}(S'(X)/R'(X))=\{MR'(X)\mid M\in\mathrm{MSupp}(S'/R')\}$ (see the proof of Corollary \ref{3.5}) and, for $M'\in\mathrm{MSupp}(S'(X)/R'(X)),\ M\in\mathrm{MSupp}(S'/R')$ such that $M'=MR'(X)$, we have $R'(X)_{M'}=R'_M(X)$ and $S'(X)_{M'}=S'_M(X)$. Then, $\ell[R'_M,S'_M]=\ell[R'(X)_{M'},S'(X)_{M'}]$, giving $\Lambda (S/R)=\Lambda (S(X)/R(X))$.
\end{proof}

\section { On some new properties of FIP extensions}

In \cite[Theorem 3.30]{DPP3}, we got the following result: Let $(R,M)$ be a quasi-local ring and $R\subset S$ a subintegral extension. Put $R_i:=R+SM^{i}$ and $M_i:=M+SM^{i}$ for each $i>0$. Then, $R(X)\subset S(X)$ has FIP if and only if $R_2 \subseteq S$ is chained and ${\mathrm L}_R((SM)/M)=n-1$, where $n:=\nu(R/(R:S))$ is the index of nilpotency of $M/(R:S)$ in $R/(R:S)$. When $|R/M|=\infty$, these conditions are equivalent to $R\subset S$ has FIP. We intend to establish a more agreeable characterization. Before that, we reprove part of \cite[Lemma 5.12]{DPP2} under weaker assumptions that are enough for our purpose.
  
\begin{lemma}\label{4.1} Let $(R,M)$ be a quasi-local Artinian ring which is not a field and let $n$ be the index of nilpotency of $M$ in $R$. Let $R\subset S$ be a finite subintegral extension such that $(R:S)=0$.  Set $R_i:=R+SM^{i}$ and $M_i:=M+SM^{i}$ for $i\in\{0,1,\ldots,n\}$. Then $R\subset S$ has FCP. Moreover,   the following conditions are equivalent:
\begin{enumerate}

\item    ${{\mathrm  L}}_R(SM/M)=n-1$.

\item  ${{\mathrm  L}}_R(M_i/M_{i+1})=1$ for all $i=1,\ldots, n-1$. 

\item$R \subseteq  R_1$ is chained.
\end{enumerate}
\end{lemma}

\begin{proof} First, we may remark that $R\subset S$ has FCP in view of \cite[Theorem 4.2]{DPP2}. Next, $(R_i,M_i)$ is quasi-local for all $i=1,\ldots, n$, because $R\subset S$ is subintegral and $R_i/M_i=(R+S M^{i})/(M+SM^{i})\cong R/[R\cap(M+SM^{i})]=R/M=:K$, which is a field. Moreover, for $1\leq i<n$, we have $M_i\neq M_{i+1}$ (for if not, we would have $SM^i\subseteq M+SM^{i+1}$ and multiplication by $M ^{n-i-1}$ would lead to $SM^{n-1}\subseteq M^{n-i}\subset R$ and $0\neq M^{n-1}\subseteq(R:S)=0$, an absurdity). It follows that $R_i\neq R_{i+1}$. Then $MR_i=M_{i+1}=(R_{i+1}:R_i)$; note also that $M_i^2\subseteq M_{i+1} \subset  M_i$.

(1) $\Leftrightarrow$ (2). Since $M_1=SM$ and $M_n=M$, we get $\sum_{i=1}^{n-1}{\mathrm  L}_R(M_i/M_{i+1})$

\noindent $={\mathrm L}_R(SM/M)$. Also, if $i=1,\ldots,n-1$, then $M_i\neq M_{i+1}$, and so ${\mathrm L}_R(M_i/M_{i+1})\geq 1$. Thus, ${\mathrm L}_R(SM/M)\geq n-1$, with equality if and only if ${\mathrm L}_R(M_i/M_{i+1})=1$ for all $i=1,\ldots, n-1$.

(2) $\Rightarrow$ (3). Assume that ${\mathrm L}_R(M_ i/M_{i+1})=1$ for all $i=1,\ldots,n-1$. Since $MM_i\subseteq M_{i+1}$ and $K=R/M$, we have ${\mathrm L}_R(M_i/M_{i+1})={\mathrm L}_{R/M}(M_i/M_{i+ 1})=\dim_K(M_i/M_{i+1})$. It follows that $\dim_K(R_i/M_{i+1})=\dim_K(R_i/M_i)+\dim_K(M_i/M_{i+1})=1+ 1=2$, and so we deduce from Theorem \ref{1.2}(c) that $R_{i+1}\subset R_i$ is a ramified (minimal) extension. We get a maximal chain $R=R_n\subset R_{n-1}\subset\, \cdots\, \subset R_2\subset R_1$. We will show that there cannot exist some $T\in[R,R_1]\setminus\{R_i\}_{i=1}^{n}$. Deny and let $k:=\max\{i\in\{1,\ldots,n-1\}\mid T\subset R_i\}$. As $T\not\subseteq R_{k+1}$, we can use FCP to find some $T'\in [T,R_k]$ such that $T'\subset R_k$ is a minimal extension. This minimal extension must be ramified because it is subintegral. Note that $T'\neq R_{k+1} $ and $M':=(T':R_k)$ is a maximal ideal of $T'$ with $ M'\cap R=M$. As $M_{k+1}=MR_k\subseteq M'R_k=M'\subset M_k$, we have $M_{k+1}\subseteq M' \subset M_k$. Since $1={\mathrm L}_R(M_k/M_{k+1})={\mathrm L}_{R/M}(M_k/M_{k+1})={\mathrm L}_ {R_k/M_k}(M_k/M_{k+1})$, the ideals $M_{k+1}$ and $M_ k$ of $R_k$ must be adjacent. Hence $M'=M_ {k+1}$. But $R_{k+1}=R+M_{k+1}=R+M'\subseteq T'\subset R_k$, and so the minimality of $R_{k+1}\subset R_k$ yields that $T'=R_{k+1}$, the desired contradiction.

(3) $\Rightarrow$ (2). In fact, we are going to show that if there exists $k\in\{1,\ldots,n-1\}$ such that ${\mathrm L}_R(M_k/M_{k+1})>1$, then $[R,R_1]$ is not linearly ordered. By \cite[ Proposition 4.7(a)]{DPP2}, we have that ${\mathrm L}_R(M_k/M_{k+1})\leq{\mathrm L}_R(M_1/M)={\mathrm L}_R(R_1/R)$ is finite. But we have ${\mathrm L}_R(M_k/M_{k+1})={\mathrm L}_{R/M}(M_k/M_{k+1})= {\mathrm L}_ {R_ k/M_k}(M_k/M_{k+1})$, which is finite. Thus, there exists an $R_k$-submodule $Q$ of $M_k$ containing $M_{k+1}$ such that ${\dim}_K(Q/M_{k+1})$

\noindent $={\dim}_K(M_k/M_{k+1})-2$. Hence ${\dim}_K(M_k/Q)=2$ and $M_k/Q$ has at least two distinct one-dimensional $K$-vector subspaces of the form $Q'/Q$ and $Q''/Q$, where $Q',Q''$ are appropriate ideals of $R_k$ that contains $Q$. Moreover, they are incomparable. Since $Q'$ and $Q''$ contain $M_{k+1}$, we have $Q'\cap R=Q''\cap R=M$. Set $T':=R+Q',\ T'':=R+Q''\subseteq R_k$. It follows that $Q'$ (resp. $Q''$) is the unique maximal ideal of $T'$ (resp. $T''$). Assume, for instance, that $T' \subset T''$. Then, $Q'\subset Q''$, a contradiction. It follows that $[R, R_1]$ is not linearly ordered.
\end{proof}

We can now offer a nicer form of \cite[Theorem 3.30]{DPP3}
  
  \begin{theorem}\label{4.3} Let $R\subset S$ be a subintegral extension. The following statements are equivalent:
\begin{enumerate}

\item  $R(X)\subset S(X)$ has FIP.

\item  $R\subset S$ has FIP and is arithmetic.
\end{enumerate}
\end{theorem}

\begin{proof} Under each statement, $R\subset S$ has FCP, so that $|\mathrm{MSupp}(S/R)|<\infty$ \cite[Theorem 3.9]{DPP3} and \cite[Corollary 3.2]{DPP2}. In view of \cite[Proposition 3.7]{DPP2}, $R\subset S$ has FIP if and only if $R_M\subset S_M$ has FIP for each $M\in\mathrm{MSupp}(S/R)$ and $ R\subset S$ has FCP. In the same way, $R(X)\subset S(X)$ has FIP if and only if $R_M(X)\subset S_M (X)$ has FIP for each $M\in\mathrm{MSupp}(S/R)$ and $R(X)\subset S(X)$ has FCP, because of \cite[Lemma 3.16]{DPP3}. It follows that we may reduce to the case where $(R,M)$ is a quasi-local ring, so that $(R(X),MR(X))$ is a quasi-local ring. In this situation, we claim that $R(X)\subset S(X)$ has FIP
if and only if  $R\subset S$ has FIP and  is   chained.

Assume first that $R(X)\subset S(X)$ has FIP. Then, $R\subset S$ has FIP by \cite[Theorem 3.30]{DPP3}. Moreover, $|R(X)/MR(X)|=|(R/M)(X)|=\infty$. Set $C':=(R(X):S(X)),\ R':=R(X)/C',\ M':=MR(X)/C'$  and $S':= S(X)/C'$. Then, $R'\subset S'$ has FIP, $R'$ is a quasi-local Artinian ring with $(R':S')=0$ and $|R'/M'|=\infty$. Assume first that $M'\neq 0$, so that $R'$ is not a field. In view of \cite[Proposition 5.15]{DPP2}, we get that $[R',S']$ is a chain. Assume now that $M'= 0$, so that $R'$ is an infinite  field. Since $ R'\subset S'$ has FIP, it follows from \cite[Theorem 3.8 and proof of Lemma 3.6]{ADM} that $[R',S']$ is  a chain. In both cases $[R',S']$ is  a chain, and so are $[R(X),S(X)]$ and $[R,S]$ by \cite[Lemma 3.1(d)]{DPP3}. 
 
Conversely, assume that $R\subset S$ has FIP and  is  chained. Set $C:=(R:S),\ R'':=R/C,\ M'':=M/C$ and $S'':=S/C$. Then, $R''\subset S''$ has FIP and is chained, $R''$ is a quasi-local Artinian ring and $(R'':S'')=0$. Assume that $R''$ is not a field. Using Lemma \ref{4.1} and its notation, we get that $[R''_2,S'']$ is  a chain, and so is $[R_2,S]$. Since $[R'',R_1'']$ is also a chain, we get that ${\mathrm L}_{R''}(S''M''/M'')=n-1$, where $n$ is the index of nilpotency of $ M''$ in $R''$. But ${\mathrm L}_{R''}(S''M''/M'')={\mathrm L}_R(S''M''/M'')= {\mathrm L}_R(SM/M)$, because of \cite[Corollary 2 of Proposition 24, page 66]{N}. Moreover, $n$ is the index of nilpotency of $ M/C$ in $R/C$. Then, we can use \cite[Theorem 3.30]{DPP3} to get that $R(X)\subset S(X)$ has FIP. Assume now that $R''$ is a field, so that $(R:S)=M$. Then, $SM=M$ gives $R_1=R_2=R$, and $n=1$ implies that ${\mathrm L}_R(SM/M)=0$ is satisfied. And \cite[Theorem 3.30]{DPP3} gives again the result.
\end{proof}

\begin{corollary}\label{4.3.1} Let $R \subseteq S$ be an FIP ring extension. Then $R(X) \subseteq S(X)$ has FIP if and only if $R \subseteq  {}^+_S R$ is arithmetic. In that case $|[R(X),S(X)]| = |[R,S]|$.
\end{corollary}
\begin{proof} Use \cite[Theorem 3.21]{DPP3} which states that $R(X) \subseteq S(X)$ has FIP if and only if $R \subseteq S$ and $R(X) \subseteq  {}^+_S R(X)$ have FIP. Conclude with \cite[Theorem 32]{DPP4}.
\end{proof}

\begin{corollary}\label{4.3.2} Let $R \subseteq S$ be an FIP ring extension such that $|R/M|=\infty$ for each $M\in\mathrm {MSupp}( {}^+_S R/R)$. Then $R(X) \subseteq S(X)$ has FIP.  The result holds in particular when $|R/M|=\infty$ for each $M\in\mathrm {MSupp}(S/R)$.
\end{corollary}

\begin{proof} It is enough to prove that  a subintegral FIP extension $R\subset S$ such that $|R/M|=\infty$ for each $M\in\mathrm {MSupp}(S/R)$ is arithmetic.  We can suppose that the conductor of $R\subseteq S$  is zero and  that $R$ is quasi-local,  with maximal ideal $M\in\mathrm{MSupp}(S/R)$.  It follows that $(R,M)$ is a quasi-local Artinian ring by \cite[Theorem 4.2]{DPP2}. Assume that $R$ is not a field. Then, $[R,S]$ is a chain by \cite[Proposition 5.15]{DPP2}. If $R$ is an infinite field, $S$ is of the form $R[\alpha]$, for some $\alpha\in S$ which satisfies $\alpha^3=0$ \cite[Theorem 3.8 (3)]{ADM}, since $R\subset S$ is subintegral. Then, $[R,S]$ is linearly ordered by the proof of \cite[Lemma 3.6 (b)]{ADM}.
\end{proof}

\begin{corollary}\label{4.3.3} Let $R\subseteq S$ be an extension, then $R(X_1,\ldots ,X_n) \subseteq S(X_1,\ldots ,X_n)$ has FIP for each integer $n\geq 0$ if and only if $R(X) \subseteq S(X)$ has FIP.
\end{corollary}

We come back to the example given in \cite[Example 3.12]{DPP3}, which shows that the arithmetic condition is necessary in Theorem \ref{4.3}.

\begin{example}\label{4.4} Let $K$ be a finite field and $T:=K[Y]/(Y^4)$. As $T$ is a finite-dimensional vector space over $K$, it follows from \cite[Theorem 3.8 (b)]{ADM} that the extension $K\subset T$ has FIP. Consider the extension $K(X)\subset T(X)$. We proved in \cite[Example 3.12]{DPP3} that $K(X)\subset T(X)$ cannot have FIP because $K(X)$ is an infinite field and $T(X)$ contains an element whose index of nilpotency is 4 since $T\to T(X)$ is injective. Another proof of this result can be given by  Theorem \ref{4.3}. Indeed, consider the coset $y:= Y+(Y^4)\in T=K[Y]/(Y^4)$. Put $S_1:=K [y^2]$ and $S_ 2:=K[y^3]$. We get that $K\subset T$ is a subintegral extension which has FIP, but $S_1$ and $S_2$ are incomparable and $K\subseteq T$ is not arithmetic. So, $K(X)\subset T(X)$ cannot have FIP.
\end{example}

\begin{remark}\label{4.5}  If  $R\subseteq S$ is not subintegral it may be that  the arithmetic condition   be superfluous. We proved that a seminormal extension $R\subseteq S$ has FIP if and only if $R(X)\subseteq S(X)$ has FIP  \cite[Corollary 3.20]{DPP3}. It is easy to exhibit  seminormal FIP  extensions $R \subset S $ with $R$ quasi-local  and $R\subseteq S$ non arithmetic (see  Example \ref{5.11}(5)). 
\end{remark}

In the next section we examine  the first  properties of  arithmetical extensions. The study will be strongly completed in a forthcoming paper.

\section{Elementary properties of arithmetical extensions}

 Using the language  and results of Knebusch and Zhang in \cite{KZ}, we are able to get a characterization of some arithmetic extensions.  We note here that  chained ring extensions $R \subseteq S$  are called $\lambda$-extensions by Gilbert \cite{G}. Knebusch and Zhang defined Pr\"ufer extensions in \cite{KZ}. It is now well known that $R\subseteq S$ is Pr\"ufer if and only if $(R,S)$ is a normal pair. We refer the reader to \cite{KZ} for the properties of Pr\"ufer extensions, noting only here that a ring extension 
 $R\subseteq S$ is Pr\"ufer if and only if $R\subseteq T$ is a flat epimorphism  for each $T\in [R,S]$. We recall some properties of a flat epimorphism $f: A \to B$ (see \cite[Chapter IV]{L}):
 
 \noindent Scholium
 
 (1) $\mathrm{Spec}(B)\to \mathrm{Spec}(A)$ is injective
 
 (2) $f$ is essential; that is, for any ring morphism $g: B\to C$, such that  $g\circ f$ is injective, then 
 $g$ is injective.
 
 (3) Each ideal $J$ of $B$ is of the form $J= f^{-1}(J)B$.
 
 (4) If $f$ is injective and $f$ is factored $A\to C \to B$, then $C\to B$ is a flat epimorphism, if it is injective.
 
 (5) The class of flat epimorphisms  is stable under base changes.
 
\noindent  We refer the reader to \cite{KZ} for the meaning of a Pr\"ufer-Manis extension, called also a PM-extension. The following proposition will be completed by Theorem \ref{5.140}.
 
\begin{proposition}\label{5.1} Let $R\subseteq S$ be an integrally closed extension. 
Then $R\subseteq S$ is arithmetic if and only if $R\subseteq S$ is locally Pr\"ufer-Manis. \end{proposition}
\begin{proof} Use  \cite[Theorem 3.1, p. 187]{KZ}
\end{proof}

\begin{proposition}\label{5.2} Let $R\subseteq S$ be an FMC extension. 
\begin{enumerate}
 \item Assume that  $R\subseteq S$ is arithmetic and integrally closed.  Then $\mathrm{Supp}_{R_P}(S_P/R_P)$ is a chain for each $P\in \mathrm{Spec}(R)$.
 \item  Assume that $R\subseteq S$ is   chained, then $|\mathrm{MSupp}_R(S/R)|=1$.
 \end{enumerate}
\end{proposition}
\begin{proof} (1) Assume that $R\subseteq S$ is an  arithmetic integrally closed FMC extension. We can assume that $R$ is local with maximal ideal $M$ in $\mathrm{Supp}(S/R)$. If $R\subseteq S$ is PM, observe that the set of all prime ideals $Q$ of $R$ such that $QS= S$  is  $\mathrm{Supp}(S/R)$  by Lemma \ref{1.8} and  is a chain by the proof of \cite[Theorem 3.1, p. 187]{KZ}. 

(2) Let $R\subset S$ be a chained  FMC extension. Let $M\in\mathrm{MSupp}(S/R)$. We begin to show that there exists $R'_1\in[R,S]$ such that $R\subset R'_1$ is a minimal extension with $\mathcal{C}(R,R'_1)=M$. Let $R=R_0\subset\cdots\subset R_i\subset\cdots\subset R_n=S$ be a maximal chain of subextensions, where $R_i\subset R_{i+1}$ is minimal for each $i\in\{0,\ldots,n-1\}$. If $\mathcal{C}(R,R_1)=M$, we set $R'_1:=R_1$. Assume that $\mathcal{C}(R,R_1) \neq M$, and set $k:=\inf\{i\in\{1,\ldots,n\}\mid\mathcal{C}(R_{i-1},R_i)= M\}$. Then, $k>1,\ \mathcal{C}(R_ {k-1},R_k)=M$, and $\mathcal{C}(R_{i-1},R_i)\neq M$ for each $i<k$. It follows that $M\not\in\mathrm{MSupp}(R_{k-1}/R)$. By \cite[Lemma 1.10]{Pic 3}, there exists $R'_1\in[R,R_k]$ such that $R\subset R'_1$ is a minimal extension with $\mathcal{C}(R,R'_1)=M$.

We claim that $|\mathrm{MSupp}(S/R)|=1$. Deny and let  $N\in\mathrm{MSupp}(S/R)$,

\noindent $N\neq M$. The previous proof shows that there exists $R''_1\in[R,S]$ such that $R\subset R''_ 1$ is a minimal extension with $\mathcal{C}(R,R''_1)=N$, so that $R''_1\neq R'_1$, a contradiction since $R\subseteq S$ is chained.
\end{proof}

We will say that a ring extension $R \subseteq S$ is {\it quasi-Pr\"ufer}  (respectively, {\it quasi-Pr\"ufer-Manis)} if $\overline R \subseteq S$ is Pr\"ufer (respectively, Pr\"ufer-Manis). 
We will also say that an extension $R \subseteq S$ is {\it pinched} at some $T\in [R,S]$ if each element of $[R,S]$ is comparable under inclusion to $T$.

\begin{proposition}\label{5.3} Let $R\subseteq S$ be an extension. Then $R\subseteq S$ is chained if and only if $R\subseteq\overline R$ is chained, $R \subseteq S$ is quasi-Pr\"ufer-Manis and $[R,S]$ is pinched at $\overline R$.
Moreover, for each invertible element $x \in S$, we have either $x \in \overline R$ or $x^{-1} \in \overline R$.
\end{proposition}
\begin{proof} Use \cite[Theorem 3.1, p. 187]{KZ} for $\overline R \subseteq S$ to prove the first statement.  We show the second. If $x \in S$ is invertible, then $\overline R[x]$ is comparable to $\overline R[x^{-1}]$ and $\overline R [x] \cap \overline R[x^{-1}] = \overline R$ because $\overline R \subseteq \overline R [x] \cap \overline R[x^{-1}]$ is integral \cite[Lemma 1.2]{G}.  
\end{proof}

\begin{remark}\label{5.4}
 We can also deduce the first statement from the second by using \cite[Theorem 3.13, p.195]{KZ} in case $\overline R \subseteq S$ is a Marot extension; that is, for each  $s \in S \setminus \overline R$, the $\overline R$-module $\overline R + \overline Rs$ is generated over $\overline R$ by a set of units of $S$.
 \end{remark}
 \begin{lemma}\label{5.5} Let $ R\subseteq S$ be an extension and $J$ an ideal of $S$ with $I: = J\cap R$. 
 
 \begin{enumerate}
 \item The map $T \mapsto T/(T\cap J)$ from $[R,S]$ to $[R/I, S/J]$ is surjective and order-preserving. Its restriction $[R+J,S] \to [R/I,S/J]$ is bijective and order-preserving and order-reflecting.
 \item If $R \subseteq S$ is  chained, then $R/I\subseteq S/J$ is  chained.
 
 \item If $R \subseteq S$ is arithmetic, then $R/I \subseteq S/J$ is arithmetic.
 \item If $R\subseteq S$ is Pr\"ufer, then  $R/I \subseteq S/J$ is a Pr\"ufer extension. In particular, if $N$  is  a maximal ideal of $S$ and $R\subseteq S$ is  chained, then $R/(N\cap R)$ is a valuation domain with quotient field $S/N$. 
 \end{enumerate}
 \end{lemma}

\begin{proof} 

To prove that (1) and (2) hold, it is enough to observe that $(R+ J)/J$ is isomorphic to $R/I$ and replacing $R$ with $R+J$, we have to work with an extension  of rings sharing the ideal $J$.  Then (3) follows from (2), because the localization at a prime ideal of $R/I$ is of the form $R_P/I_P$, where $P$ is a prime ideal of $R$, and $J_P \cap R_P= I_P$.

Then (4) is a consequence of the following facts: $R\subseteq S$ is Pr\"ufer entails that $R+J\subseteq S$ is Pr\"ufer and then it is enough to use \cite[Proposition 5.8, p.52]{KZ}.
 
For the last statement, use Proposition~\ref{5.3}, because $R/(N\cap R)\subseteq S/N$ is  chained by (2). \end{proof}
\begin{remark}\label{5.5.1} It follows from Lemma~\ref{5.5} that  a quasi-Pr\"ufer extension $R\subseteq S$ gives a quasi-Pr\"ufer extension  $R/(J\cap R) \subseteq S/J$ for each ideal $J$ of $S$ and $\overline{R/(J \cap R)} =\overline R/(\overline R \cap J)$.
\end{remark}

Let $U$ be an absolutely flat ring. Recall that   each element $x$ of $U$ has a unique quasi-inverse $x'\in U$,  defined by   $x^2x' =x$ and $x'^2x =x'$. In that case,  set $e= xx'$. Then $e$ is an idempotent and $1-e +x$ is a unit of $U$, such that $(1-e+x)^{-1}= (1-e +x')$.

\begin{proposition}\label{5.6} Let $R\subseteq S$ be a chained ring extension, such that  $S$ is zero-dimensional.  
\begin{enumerate}
\item  $S\cong \mathrm{Tot}(R)$ and then $\overline R$ is a Pr\"ufer ring. 

\item   Each $x\in S/\mathrm{Nil}(S)$ has a quasi-inverse $x'\in S/\mathrm{Nil}(S)$, such that either $x$ or $x'$ belongs to $\overline R/\mathrm{Nil}(R)$. 

\item $\overline R \subseteq S$ is additively regular, whence a Marot extension.

\end{enumerate}
\end{proposition}
\begin{proof} We observe that $\overline R\subseteq S$ is  chained and then $\overline R \subseteq S$ is Pr\"ufer by Proposition~\ref{5.3}. It follows from \cite[Corollaire 4]{Z}, that  $S$ identifies with $\mathrm{Tot}(R)$ and hence $\overline R$ is a Pr\"ufer ring. Since $\overline R \subseteq S$ is integrally closed, we have that $\mathrm{Nil}(S) =\mathrm{Nil}(\overline R)$. Set $U:= S/\mathrm{Nil}(S)$ and $T:= \overline R/\mathrm{Nil}(\overline R)$. We get a Pr\"ufer extension $T\subseteq U$ by Lemma~\ref{5.5}, where $U$ is absolutely flat, whence $T\subseteq U$ is integrally closed. By the above recall and Proposition~\ref{5.3}, if $x$ is in $U$, then either $x\in T$ or $x'\in T$ because $1-e$ is an idempotent of $U$, belonging to $T$. Moreover, there is some $t= 1-e\in T$ such that $x+t$ is invertible in $U$. Since $\mathrm{Nil}(S) =\mathrm{Rad}(S)$, the Jacobson radical, we get that the same property holds for the extension $\overline R \subseteq S$. In other words, $\overline R \subseteq S$ is additively regular, whence a Marot extension (see \cite[Remark 3.15, p. 196]{KZ}).
\end{proof}

Gilbert proved that an integral domain $R$ with quotient field $K$ is such that $R\subseteq K$ is chained ($R$ is a $\lambda$-domain with the Gilbert's terminology) if and only if $[R,K] $ is pinched at $\overline R$, $R$ is a quasi-local $i$-domain and $R\subseteq\overline R$ is chained \cite[Theorem 1.9]{G}. We note that this result implies  that $R$ is a quasi-local unbranched domain, that is $\overline R$ is quasi-local (actually, in this case $\overline R$ is a valuation domain).

 We intend to generalize this result to some extension.
Before that we give a characterization of $i$-{\it pairs,} that are ring extension $R\subseteq S$ such that $\mathrm{Spec}(T) \to \mathrm{Spec}(R)$ is injective for each $T\in [R,S]$. We will say that a quasi-local ring $R$ is {\it unbranched }in $S$ if $\overline R$ is quasi-local.

\begin{proposition} \label{5.7} An extension $R \subseteq S$ defines an $i$-pair if and only if $R \subseteq S$ is quasi-Pr\"ufer and $R \subseteq \overline R$ is spectrally injective.
\end{proposition}
\begin{proof} One implication is given by \cite[Theorem 5.2(9), p. 47]{KZ}. For the converse, assume that  $\mathrm{Spec}(\overline R) \to \mathrm{Spec}(R)$ is injective  and that $R \subseteq S$ is quasi-Pr\"ufer and let $T \in [R,S]$. To conclude, consider $U:= \overline RT$. Then $\overline R  \subseteq U$ is a flat epimorphism, whence spectrally injective and $T\subseteq U$ is integral. Since $R \subseteq U$ is spectrally injective, we get that $R\subseteq T$ is spectrally injective.
\end{proof}

\begin{remark}\label{5.8} A similar proof shows that an extension $R \subseteq S$ is quasi-Pr\"ufer if and only if $R\subseteq S$ is an Inc-pair.
\end{remark}

\begin{proposition}\label{5.9} Let $R\subseteq S$ be an extension, such that $R$ is quasi-local and unbranched in $S$.  Then $R\subseteq S$ is chained if and only if $R\subseteq\overline R$ is chained, $ R\subseteq S$ is quasi-Pr\"ufer and $[R,S]$ is pinched at $\overline R$. In that case $R \subseteq S$ defines an $i$-pair.
\end{proposition}
\begin{proof}  We can suppose that $\overline R\neq S$. Observe that $\overline R$ is quasi-local. So $\overline R \subseteq S$ is Pr\"ufer if and only it is Pr\"ufer-Manis \cite[Theorem 1.8, p. 181]{KZ} and also, if and only if $\overline R\subseteq S$ is chained \cite[Theorem 3.1, p. 187]{KZ}. The first  statement is now clear.

Now, since $\overline R$ is quasi-local, from \cite[Theorem 6.8]{DPP2}, we deduce that there exists $Q \in \mathrm{Spec}(\overline R)$ such that $S= \overline R_Q$, $Q= SQ$ and $\overline R/Q$ is a valuation domain. Under these conditions $S/Q$ is the quotient field of $\overline R$/Q and $Q$ is a divided ideal of $\overline R$; that is, comparable with any other prime ideal of $\overline{R}$. We observe that $Q$ is the conductor of $\overline R \subseteq S$. Let $M, M'$ be two prime ideals of $\overline R$ lying over some prime ideal $P$ of $R$. If $M$ and $M'$ both contain $Q$, they are comparable and by incomparability of $ R \subseteq \overline R$, we get that $M= M'$. If $M\subseteq Q \subseteq M'$, we get also $M=M'$. Thus there is only one case to examine: $M, M' \subset Q$. Since the flat  extension $\overline R \subseteq S$ has the Going-Down property, $Q$ is a minimal prime ideal in $\overline R$  and then  $M= M'$. To conclude, it is enough to use Proposition~\ref{5.7}, because $R \subseteq \overline R$ is spectrally injective.
  \end{proof}
  The following ``birationnal'' result is surely well-known.
  \begin{lemma}\label{5.9.1} Let $R$ be a ring whose total quotient ring $S$ is zero-dimensional and with integral closure  $\overline R$. Then the  map $\mathrm{Spec}(S) \to \mathrm{Spec}(R)$ is injective and induces bijective maps $\varphi : \mathrm{Max}(S)= \mathrm{Min}(S) \to \mathrm{Min}(R)$ and $\psi :\mathrm{Min}(\overline R) \to \mathrm{Min}(R)$.  For $M \in \mathrm{Min}(R)$, we  set $M_S : =\varphi^{-1}(M)$ and $M_{\overline R} := \psi^{-1}(M)$.
 \end{lemma}
 \begin{proof} For each injective extension $A \subseteq B$, any minimal prime ideal of $A$ is lain over by a minimal prime ideal of $B$, any minimal prime ideal of $B$ contracts to a minimal prime ideal of $A$ when $A \subseteq B$ is flat and $\mathrm{Spec}(B) \to \mathrm{Spec}(A)$ is injective when $A \subseteq B$ is a flat epimorphism (see Scholium).
 \end{proof}

\begin{theorem}\label{5.10} Let $R\subseteq S$ be an extension, where $R$ is locally irreducible and  $S$ is zero-dimensional.  
\begin{enumerate}

\item If $R\subseteq S$ is  chained, then $\overline R$ is a Pr\"ufer ring with total quotient ring $S$ and $R \subseteq S$ defines an $i$-pair. Moreover, the following two conditions (*) and (**) hold: 

(*)    $R$ is locally unbranched in $S$.

(**)  $R/M$ is a quasi-local $i$-domain for each $M\in \mathrm{Min}(R)$.

\item Suppose that $R\subseteq\overline R$ is chained, $[R,S]$ is pinched at $\overline R$, $S =\mathrm{Tot}(R)$ and $\overline R$ is Pr\"ufer.

(a) If  (*) holds, then  $R\subseteq S$ is arithmetic.

(b)  If (**) holds,  then $R/M\subseteq S/ M_S$ is chained, for each $M\in \mathrm{Min}(R)$.
\end{enumerate}
\end{theorem}
\begin{proof}  We first prove (1) and suppose that $R\subseteq S$ is chained. Then $R \subseteq S$ is quasi-Pr\"ufer-Manis by Proposition~\ref{5.3}. Hence $S$ can be identified to $\mathrm{Tot}(R)$ in view of Proposition~\ref{5.6} and $\overline R$ is a Pr\"ufer ring. 

Moreover,  $\mathrm{Spec}(\overline R) \to \mathrm{Spec}(R)$ induces a bijection $\mathrm{Min}(\overline R) \to \mathrm{Min}(R)$ by Lemma~\ref{5.9.1}.

We  claim that $\mathrm{Spec}(\overline R) \to \mathrm{Spec}(R)$ is injective. Let $M, N$ be two prime ideals of $\overline R$ lying both over a prime ideal $P$ of 
$R$ and let $\mathfrak{P}$ be the unique minimal prime ideal of 
$R$  contained in $P$. A minimal prime ideal $\mathfrak{M}$ of $R,\ \mathfrak{M}\subseteq M$ necessarily  lies over $\mathfrak{P}$. It follows then that $\mathfrak{M}= \mathfrak{P}_{\overline R}$ is contained in $N$.  Since $\mathfrak{P}_S$ is maximal, Lemma~\ref{5.5} shows that $\overline R/\mathfrak{P}_{\overline R}$ is a valuation domain and then $\mathrm{Spec}(\overline R/\mathfrak{P}_{\overline R})$ is a chain. The preceding observations yield that $\mathrm{Spec}(\overline R) \to \mathrm{Spec}(R)$ is injective, because $R  \subseteq \overline R$ is an Inc-extension. Therefore, $R\subseteq S$ defines an $i$-pair by Proposition~\ref{5.7}.  From Lemma~\ref{5.5} and Remark~\ref{5.5.1}, we deduce  that $R/\mathfrak{P}$  is a quasi-local  $i$-domain with  integral closure $\overline R/\mathfrak{P}_{\overline R}$ and quotient field $S/\mathfrak P_S$ (see \cite[Proposition 2.14]{PA}). 

We now prove (2). (a) is a consequence of Proposition~\ref{5.9}, since for each multiplicatively closed subset $\Sigma$ of $R$, the map $T \mapsto T_\Sigma$ is a surjection from $[R,S]$ to $[R_\Sigma,S_\Sigma]$. Then (b) follows also from Proposition~\ref{5.9}. 
 \end{proof}
 
 In case $R$ is an integral domain, we recover in (2)(b) the Gilbert's above-mentioned result.

\begin{example}\label{5.11} Arithmetic extensions appear frequently, as the reader may see below.

(1) An integrally closed FCP (whence FIP) extension $R\subseteq S$ is arithmetic. Indeed,  $R_M\subset S_M$ is  integrally closed and FCP  for each $M\in\mathrm{MSupp}(S/R)$, so that  $[R_M,S_M]$ is a chain  \cite[Theorem 6.10]{DPP2}. 

(2) A subintegral FIP extension $R\subset S$ such that $|R/M|=\infty$ for each $M\in\mathrm {MSupp}(S/R)$ is arithmetic.  We already proved this result in the proof of Corollary~\ref{4.3.2}.

(3) For a  t-closed FIP integral extension $R\subseteq S$, Lemma \ref{3.1}(3) makes sense to say  that $R_M/MR_M\subset S_M/MR_M$ is a purely inseparable field extension for each $M\in\mathrm{MSupp}(S/R)$. We assume that these hypotheses hold and show that $R\subseteq S$ is arithmetic. 

We can reduce to the case where $R$ is local with maximal ideal $M:= (R:S)$. Then $[R/M,S/M]$ is a chain by \cite [Proposition 2, Ch. V, page 24]{Bki A} and so is $[R,S]$.

(4) Let $R\subset S$ be an FIP extension. Assume that $R_M\subset S_M$ satisfies one of the above conditions  (1), (2) or (3) for each $M\in\mathrm{MSupp}(S/R)$. Then $R\subseteq S$ is arithmetic.

(5) On the contrary,  a seminormal and infra-integral FIP extension $R\subset S$ is never arithmetic.  To see this, we can suppose that $R$ is quasi-local with maximal ideal $M\in\mathrm{MSupp}(S/R)$  and $(R:S)= M$ by using a suitable localization. Using the proof of \cite[Proposition 4.16]{DPP3}, we get that $ S/M\cong(R/M)^n$ for some positive integer $n$ and then $[R,S]$ is not  a chain. 

(6) It may be asked when is a field extension $K \subseteq F $  arithmetic (chained)? To the authors knowledge, the only comprehensive study about the question is given in \cite{V}, from which we extract the following. An intermediary extension $L$ of $K \subseteq F$ is called {\it reduced} if $L\neq F$ and for all $c,  d \in F\setminus L$, $L(c) =L(d) \Rightarrow K(c)= K(d)$. Then $[K,L]$ is a chain if and only if each of the elements of $[K,L]\setminus \{F\}$ is reduced. In this case $K\subseteq F$ is algebraic. If $K\subseteq F$ is finite and Galois, with Galois group $G$, then $K\subseteq F$  is arithmetic if and only if either $G$ is cyclic of order $p^n$ ($p$ a prime number and $n$ an integer $> 0$) or $G$ is isomorphic to a generalized quaternion group of order $2^n, n\geq 3$ and in this case $[K,F] =\{K,L,F\}$, with $[F:L]= 2$.
Other criteria are given for separable finite extensions. Note also that if $[K,L]$ is  a chain and $K\subseteq L$ algebraic, then $K\subseteq F$ is either separable or purely inseparable.
\end{example}

 Olberding in \cite{O} says that an extension of rings $R \subseteq S$ is {\it quadratic}  if each intermediate $R$-submodule of $S$ containing $R$ is a ring. Other authors call  $\Delta_0$-{\it extension} such extensions and we will follow them. An extension $R \subseteq S$ is called {\it quadratic} if each $s\in S$ satisfies  $P(s)= 0$ for a monic quadratic polynomial $P(X) \in R[X]$ (see for instance \cite{HP}). We  call an extension $R\subseteq S$ a $\Delta$-{\it extension} if $[R,S]$ is stable under addition, that is  $T_1+T_2 =T_1T_2$ for $T_1, T_2 \in [R,S]$. Note that  an extension $R\subseteq S$ is a $\Delta_0$-extension if and only if it is a quadratic $\Delta$-extension and also that these properties localize and globalize. Actually, the proofs of \cite{HP} given for integral domains are valid for arbitrary extensions.
 
 We first give some examples of $\Delta_0$-extensions.
 
 \begin{proposition}\label{5.12} Let $R \subseteq S$ be a spectrally injective integral  (for example, subintegral) FCP extension of rings. If  the $R$-module $S/R$ is locally uniserial (for example when $R\subseteq S$ is locally minimal), then $R\subseteq S$ is an arithmetic $\Delta_0$-extension.
\end{proposition}
\begin{proof}
 We can assume that $R \subseteq S$ is an integral FCP extension of rings $R \subseteq S$, which is spectrally injective, with $R$ quasilocal and assume that the $R$-module $S/R$ is uniserial. Since $\mathrm{Spec}(S) \to \mathrm{Spec}(R)$ is injective, $S$ is quasilocal. Moreover, $S/R$ is an Artinian $R$-module because $R/(R:S)$ is Artinian  (\cite[Theorem 4.2]{DPP2}) and $S$ is an $R$-module of finite type. It follows from \cite[Lemma 4.1]{O} that $R \subseteq S$ is a $\Delta_0$-extension.
 \end{proof}

\begin{proposition}\label{5.13} Let $R \subseteq S$ be a $\Delta$-extension. Then $\overline R \subseteq S$ is a Pr\"ufer extension.
\end{proposition}
\begin{proof}
 It is enough to apply  \cite[Theorem 1.7, p. 88]{KZ}.
\end{proof}

\begin{proposition}\label{5.14} An arithmetic extension $R \subseteq S$  is a $\Delta$-extension and  hence  is quasi-Pr\"ufer.
\end{proposition}
\begin{proof}  In case $[R,S]$ is  a chain, we have $B+C= BC= \mathrm{Max}(B,C)$ for $B,C$ in $[R,S]$. 
\end{proof}

\begin{theorem}\label{5.140} Let $R \subseteq S$ be an integrally closed extension, then $R \subseteq S$ is arithmetic if and only if $R\subseteq S$ is Pr\"ufer and,  if and only if $R\subseteq S$ is locally Pr\"ufer-Manis.
\end{theorem}
\begin{proof}  Assume that $R \subseteq S$ is integrally closed. If  $R \subseteq S$ is arithmetic, then $R\subseteq S$ is Pr\"ufer. Conversely, if $R\subseteq S$ is Pr\"ufer, then $R\subseteq S$ is arithmetic. It is enough to use Proposition~\ref{5.1}  and \cite[Theorem 5.1, p. 46]{KZ} which states that $R\subseteq S$ is locally Pr\"ufer-Manis if $R\subseteq S$ is a Pr\"ufer extension.
\end{proof}

\begin{proposition}\label{5.15} Let $R\subseteq S$ be an arithmetic extension. Then  for $B, C, D \in [R,S]$, we have  $B\cap(C.D) = (B\cap C).(B\cap D)$ and $B.(C\cap D) = (B.C)\cap (B.D)$. Hence $([R,S], \cap,. )$ is a  complete modular lattice.
\end{proposition}  
\begin{proof} These equalities are locally  trivial.
\end{proof}

These distributivity properties do not imply that the extension is arithmetic. See Remark \ref{5.16}.

\begin{remark}\label{5.16}   
 Consider the following example. Set $R:=\mathbb{Q},\ T_1:=\mathbb{Q}(\sqrt 2),\ T_2:=\mathbb{Q}(\root 3 \of 2)$ and $S:=\mathbb{Q}(\root 6 \of 2)$. Then, setting $z:=\root 6 \of 2,\ x:=\root 3 \of 2$ and $y:=\sqrt 2$, so that $x=z^ 2$ and $y=z^3$, we get $S=\mathbb{Q}(z),\ T_1=\mathbb{Q}(y)$ and $T_2=\mathbb{Q}(x)$. The (field) extensions $R\subset T_i$ and $T_i\subset S$, for $i=1,2$ are all minimal inert (ring) extensions with crucial ideal $0$. 
 
 Using the proof of the Primitive Element Theorem (see \cite[Ch. V, Th\' eor\`eme 1, p. 39]{Bki A}), we get that $[R,S]=\{R,T _1,T_2,S\}$, so that $([R,S],\cdot,\cap)$, is a 
 complete modular lattice, since $T_1T_2=S$ and $T_1\cap T_2=R$. Indeed, the minimal polynomial of $z$ is $X^6-2$, whose divisors in $S[X]$ are $X-z,X+z,X^2-z^2=X^2-x, X^3-z^3=X^3-y,X^3+z^3=X^3+y,X^2+zX+z^2,X^2-zX+z^2,X^3-2zX^2+2z^2X-z^3,X^3+2zX^2+2z^2X+z^3$. However, $[R,S]$ is not  a chain. 
\end{remark}

\end{document}